\documentclass[12pt]{amsart}
\usepackage{amscd,amsmath,amsthm,amssymb,graphics}
\usepackage{pstcol,pst-plot,pst-3d}
\usepackage{lmodern,pst-node}
\usepackage{multicol}
\usepackage{epic,eepic}
\usepackage{amsfonts,amssymb,amscd,amsmath,enumerate,verbatim}
\usepackage[all]{xy}
\usepackage{lineno}
\psset{unit=0.7cm,linewidth=0.8pt,arrowsize=2.5pt 4}

\newpsstyle{fatline}{linewidth=1.5pt}
\newpsstyle{fyp}{fillstyle=solid,fillcolor=verylight}
\definecolor{verylight}{gray}{0.97}
\definecolor{light}{gray}{0.9}
\definecolor{medium}{gray}{0.85}
\definecolor{dark}{gray}{0.6}

\def\NZQ{\Bbb}               

\def\ZZ{{\NZQ Z}}

\def\FF{{\NZQ F}}
\def\GG{{\NZQ G}}

\def\DD{{\NZQ D}}
\def\frk{\frak}               

\def\Phi{{\frk n}}
\def\Phi{{\frk N}}
\def\MS{{\mathcal S}}
\def\MC{{\mathcal C}}

\def\ab{{\bold a}}
\def\bb{{\bold b}}

\def\eb{{\bold e}}

\def\opn#1#2{\def#1{\operatorname{#2}}} 
\opn\chara{char} \opn\length{\ell} \opn\pd{pd} \opn\rk{rk}
\opn\projdim{proj\,dim} \opn\injdim{inj\,dim} \opn\rank{rank}
\opn\depth{depth} \opn\grade{grade} \opn\height{height}
\opn\embdim{emb\,dim} \opn\codim{codim}

\opn\Tr{Tr} \opn\bigrank{big\,rank}
\opn\superheight{superheight}\opn\lcm{lcm}
\opn\trdeg{tr\,deg}
\opn\reg{reg} \opn\lreg{lreg} \opn\ini{in} \opn\lpd{lpd}
\opn\size{size}\opn\bigsize{bigsize}
\opn\cosize{cosize}\opn\bigcosize{bigcosize}
\opn\sdepth{sdepth}\opn\sreg{sreg}
\opn\link{link}\opn\fdepth{fdepth}\opn\lin{lin} \opn\ini{in}
\opn\div{div} \opn\Div{Div} \opn\cl{cl} \opn\Cl{Cl}
\opn\Spec{Spec} \opn\Supp{Supp} \opn\supp{supp} \opn\Sing{Sing}
\opn\Ass{Ass} \opn\Min{Min}\opn\Mon{Mon} \opn\dstab{dstab} \opn\astab{astab}
\opn\Syz{Syz}
\opn\Ann{Ann} \opn\Rad{Rad} \opn\Soc{Soc}
\opn\Im{Im} \opn\Ker{Ker} \opn\Coker{Coker} \opn\Am{Am}
\opn\Hom{Hom} \opn\Tor{Tor} \opn\Ext{Ext} \opn\End{End}
\opn\Aut{Aut} \opn\id{id}

\opn\nat{nat}
\opn\pff{pf}
\opn\Pf{Pf} \opn\GL{GL} \opn\SL{SL} \opn\mod{mod} \opn\ord{ord}
\opn\Gin{Gin} \opn\Hilb{Hilb}\opn\sort{sort}
\opn\initial{init}
\opn\ende{end}
\opn\height{height}
\opn\type{type}
\opn\aff{aff} \opn\con{conv} \opn\relint{relint} \opn\st{st}
\opn\lk{lk} \opn\cn{cn} \opn\core{core} \opn\vol{vol}
\opn\link{link} \opn\star{star}\opn\lex{lex}\opn\sign{sign}
\opn\gr{gr}

\def\pot#1#2{#1[\kern-0.28ex[#2]\kern-0.28ex]}

\opn\dirlim{\underrightarrow{\lim}}
\opn\inivlim{\underleftarrow{\lim}}
\let\union=\cup

\let\dirsum=\oplus
\let\tensor=\otimes
\let\iso=\cong

\let\Dirsum=\bigoplus

\let\to=\rightarrow
\let\To=\longrightarrow
\def\Implies{\ifmmode\Longrightarrow \else
        \unskip${}\Longrightarrow{}$\ignorespaces\fi}
\def\implies{\ifmmode\Rightarrow \else
        \unskip${}\Rightarrow{}$\ignorespaces\fi}
\def\iff{\ifmmode\Longleftrightarrow \else
        \unskip${}\Longleftrightarrow{}$\ignorespaces\fi}

\let\:=\colon
\newtheorem{Theorem}{Theorem}[section]
 \newtheorem{Lemma}[Theorem]{Lemma}
 \newtheorem{Corollary}[Theorem]{Corollary}

\let\epsilon\varepsilon
\let\kappa=\varkappa
\def\qed{\ifhmode\textqed\fi
      \ifmmode\ifinner\quad\qedsymbol\else\dispqed\fi\fi}
\def\textqed{\unskip\nobreak\penalty50
       \hskip2em\hbox{}\nobreak\hfil\qedsymbol
       \parfillskip=0pt \finalhyphendemerits=0}
\def\dispqed{\rlap{\qquad\qedsymbol}}

\opn\dis{dis}
\def\pnt{{\raise0.5mm\hbox{\large\bf.}}}

\opn\Lex{Lex}

\begin{document}
 \title{The linear strand of determinantal facet ideals}

 \author {J\"urgen Herzog, Dariush Kiani  and Sara Saeedi Madani}

\address{J\"urgen Herzog, Fachbereich Mathematik, Universit\"at Duisburg-Essen, Campus Essen, 45117
Essen, Germany} \email{juergen.herzog@uni-essen.de}

\address{Dariush Kiani, Department of Pure Mathematics,
 Faculty of Mathematics and Computer Science,
 Amirkabir University of Technology (Tehran Polytechnic),
424, Hafez Ave., Tehran 15914, Iran, and School of Mathematics, Institute for Research in Fundamental Sciences (IPM),
P.O. Box 19395-5746, Tehran, Iran.} \email{dkiani@aut.ac.ir, dkiani7@gmail.com}

\address{Sara Saeedi Madani, Universit\"at Osnabr\"uck, Institut f\"ur Mathematik, 49069 Osnabr\"uck, Germany}
\email{sara.saeedimadani@uni-osnabrueck.de, sarasaeedim@gmail.com}

\thanks{}
\thanks{}

 \begin{abstract}
Let $X$ be an $(m\times n)$-matrix of indeterminates, and let $J$ be the ideal generated by a set $\MS$ of maximal minors of $X$. We construct the linear strand of the resolution of $J$. This linear strand  is determined by the clique complex of the $m$-clutter corresponding to  the set $\MS$. As a consequence  one obtains explicit formulas for the graded Betti numbers $\beta_{i,i+m}(J)$ for all $i\geq 0$. We also determine all sets $\MS$ for which $J$ has a linear resolution.
 \end{abstract}

\thanks{The third author was supported by the German Research Council DFG-GRK~1916.}


\maketitle

\section*{Introduction}
In this paper we consider ideals generated by an arbitrary  set of maximal minors of an $(m\times n)$-matrix of indeterminates. Such ideals have first been considered when $m=2$, in which case the set of minors (which is a set of binomials) is in bijection with the edges of a graph $G$,  and therefore is called the  binomial edge ideal of $G$. This class of ideals has first been considered in \cite{HHHKR} and \cite{O}. In \cite{HHHKR} the relevance of such ideals for algebraic statistics has been stressed. In the sequel binomial edge ideals have been studied in numerous papers with an attempt to better understand their algebraic and homological properties, see for example \cite{EHH}, \cite{EHH1}, \cite{KS1}, \cite{SZ} and \cite{ZZ}. In some special cases, the resolution of such ideals has been determined and upper bounds for their regularity have been given, see \cite{D}, \cite{EZ}, \cite{KS}, \cite{KS2}, \cite{MM}, \cite{SK} and \cite{SK1}.

When $m>2$, these ideals are called determinantal facet ideals. Here its generators are in bijection to the facets of a pure simplicial complex of dimension $m-1$. These ideals were introduced and first studied in \cite{EHHM}. About the resolution of determinantal facet ideals is known even less than for binomial edge ideals. Apart from a very special case considered in \cite{Mo}, the resolution of a determinantal facet ideal is only known when the underlying simplicial complex is a simplex, in which case the Eagon-Northcott complex provides a resolution.

One of the motivations for writing this paper was a conjecture made in \cite{KS} by the second and third author of this paper.
Given a finite simple graph $G$. Let $J_G$ be its binomial edge ideal. Then the graded Betti numbers $\beta_{i,i+2}(J_G)$ give the ranks of the corresponding free modules of the linear strand of $J_G$. Now the conjecture made in \cite{KS} says that $\beta_{i,i+2}(J_G)=(i+1)f_{i+1}(\Delta(G))$ for all $i\geq 0$. Here $\Delta(G)$ is the  clique complex of $G$ and $f_{i+1}(\Delta(G))$ is the number of faces of $\Delta(G)$ of dimension $i+1$. In this paper we will not only prove this conjecture but also prove an analogue statement for determinantal facet ideals, see Corollary~\ref{linearbetti}. In fact, the linear strand of any determinantal facet ideal, which is a subcomplex of its minimal graded free resolution, is explicitly described in Theorem~\ref{hope}.

In general, the linear strand of a graded minimal free resolution is a linear complex, that is, a complex whose matrices describing the differentials of the complex have linear forms as entries. Now given any finite linear complex, one may ask when such a complex is the linear strand of the graded minimal free resolution of a suitable graded module. This question, which is of interest by itself,  is answered in the first section of the paper, see Theorem~\ref{when}. A different description of the linear strand is given by Eisenbud in \cite[Theorem~7.4]{Ei1}. Our characterization of linear strands is then used to prove that the complex constructed in Section~\ref{Generalized Eagon-Northcott} is indeed the linear strand of the corresponding determinantal facet ideal.

In order to describe this complex,  let $\Delta$ be a simplicial complex on the vertex set $[n]$ and let $\varphi$ be a linear map between free modules of rank $m$ and $n$.  We define a subcomplex of the Eagon-Northcott complex $\MC(\varphi)$ associated with  $\varphi$ which we denote by $\MC(\Delta;\varphi)$. The complex  $\MC(\Delta;\varphi)$  is obtained from $\MC(\varphi)$ by restriction to basis elements determined by the faces of $\Delta$. In order to see that this is indeed a well defined subcomplex one has to describe the  differentials of $\MC(\varphi)$ explicitly in terms of the given natural bases.   For the convenience of the reader we included this description of the Eagon-Northcott complex. We call the subcomplex $\MC(\Delta;\varphi)$ the generalized Eagon-Northcott complex (determined by $\Delta$ and $\varphi$). When $m=1$, $\mathcal{C}(\Delta;\varphi)$ is a very special cellular complex. Cellular complexes were considered by Bayer, Peeva and Sturmfels (\cite{BS}, \cite{BPS}). Fl{\o}ystad in \cite{Fl} considered this cellular complex and called their homology to be the enriched homology module of $\Delta$ with respect to $K$.

In Theorem~\ref{missing} it is shown that $\MC(\Delta;\varphi)$ is the linear strand of a module with initial degree $m$ if and only if $\Delta$ has no minimal nonfaces of cardinality $\geq m+2$. Here and in all the following statements $\varphi$ is defined by an $(m\times n)$-matrix $X$ of indeterminates. We use this result in Section~\ref{determinantal} to obtain one of the main results of this paper: let $C$ be an $m$-uniform clutter, that is a collection of subsets of $[n]$ of cardinality $m$. The elements of $C$ are called the circuits of $C$. The clique complex of $C$, which we denote by $\Delta(C)$, has as faces all subsets $\sigma$  of $[n]$ with the property that each $m$-subset of $\sigma$  belongs to $C$. The clique complex $\Delta(C)$ has so minimal nonfaces of cardinality $\geq m+2$, and hence Theorem~\ref{missing} can be applied, and it is shown in Theorem~\ref{hope} that $\MC(\Delta(C);\varphi)$ is the linear strand of the resolution of $J_C$ where $J_C$ is generated by all maximal minors of $X$ whose columns are determined by the circuits of $C$. From this fact one deduces immediately the formula for the Betti numbers of the linear strand of $J_C$, namely
\[
\beta_{i,i+m}(J_C)={m+i-1\choose m-1} f_{m+i-1}(\Delta(C)).
\]

In Section~\ref{linear resolution} the determinantal facet ideals with linear resolution are characterized, and it is shown in Theorem~\ref{linear res} that $J_C$ has linear resolution if and only if $J_C$ is linearly presented, and that this is the case if and only if the clutter $C$ is complete, which means that $C$ consists of all $m$-subsets of a given set $V\subset [n]$.

\section{The linear strand of a graded free resolution}\label{linear strand}

Let $K$ be a field and $S=K[x_1,\ldots,x_n]$ be the polynomial ring over $K$ in the indeterminates $x_1,\ldots,x_n$. We view $S$ as a standard graded $K$-algebra by assigning to each $x_i$ the degree $1$. Let $M$ be a finitely generated graded $S$-module. Let $d$ be the initial degree of $M$, namely, the smallest integer $i$ such that $M_i\neq 0$.

Let $(\FF,\partial)$ be the minimal graded free resolution of $M$. Then $F_i=\Dirsum_jS(-j)^{\beta_{i,j}}$ where $\beta_{ij}=\dim_K\Tor_i(K,M)_j$ is the $ij$-th graded Betti number of $M$. Note that $\beta_{ij}=0$ for all pairs $(i,j)$ with  $j<i+d$.

Let $F_i^{\lin}$ be the direct summand $S(-i-d)^{\beta_{i,i+d}}$ of $F_i$. It is obvious that  $\partial (F_i^{\lin})\subset F_{i-1}^{\lin}$ for all $i>0$. Thus
\[
\FF^{\lin}\:\; \cdots\To F_2^{\lin} \To F_1^{\lin}\To F_0^{\lin}\To 0
\]
is a subcomplex of $\FF$, called the {\em linear strand} of the resolution of $M$. Throughout the paper, by the linear strand of a module $M$ we mean the linear strand of the resolution of $M$.

\medskip
A graded complex $\GG:   \cdots\to G_2 \to G_1\to G_0\to 0$ of finitely generated graded free $S$-modules is called a {\em linear complex} ({\em with  initial degree $d$}) if for all $i$, $G_i=S(-i-d)^{b_i}$ for suitable integers $b_i$.  Obviously, $\FF^{\lin}$ is a linear complex.

The question arises which linear complexes are the linear strand of a finitely generated graded module. The following result answers this question.

\begin{Theorem}
\label{when}
Let $\GG$ be a finite linear complex with initial degree $d$. The following conditions are equivalent:
\begin{enumerate}
\item[{\em (a)}] $\GG$ is the linear strand of a finitely generated graded $S$-module with initial degree~$d$.
\item[{\em (b)}] $H_i(\GG)_{i+d+j}=0$ for all $i>0$ and for $j=0,1$.
\end{enumerate}
\end{Theorem}

\begin{proof} We may assume that $d=0$.

(a)\implies (b): Let $\GG=\FF^{\lin}$ where $\FF$ is the minimal graded free resolution of a finitely generated graded $S$-module $M$. We denote by $\DD$ the quotient complex $\FF/\FF^{\lin}$. From the exact sequence of complexes
\[
0\To \FF^{\lin}\To \FF \To \DD\To 0
\]
we deduce that
\begin{eqnarray}
\label{iso}
H_i(\FF^{\lin})_{i+j}\iso H_{i+1}(\DD)_{i+1+(j-1)}
\end{eqnarray}
for all $i>0$ and all $j$.

Since $(D_i)_j=0$ for all $i\geq 0$ and all $j\leq i$, we see that $H_{i+1}(\DD)_{i+1+(j-1)}=0$ for $j=0$ and $j=1$. Thus (\ref{iso}) yields the desired conclusion.

(b)\implies(a): Let $M$ be a graded module whose minimal graded free resolution $\FF$ satisfies the condition that $F_1^{\lin}\to F_0^{\lin}$ is isomorphic to $G_1\to G_0$. (For example we could choose $M=H_0(\GG)$). We prove by induction on $i$ that the truncated complex
\[
G_i\to G_{i-1}\to\cdots \to G_1\to G_0\to 0
\]
is isomorphic to the truncated complex
\[
F_i^{\lin}\to F_{i-1}^{\lin}\to \cdots \to F_1^{\lin}\to F_0^{\lin}\to 0.
\]
By the definition of $M$ this is the case for $i=1$. Now let $i>1$. By induction hypothesis we have
\[
\Ker(G_i\to G_{i-1})\iso\Ker(F_i^{\lin}\to F_{i-1}^{\lin}).
\]
Let $Z=\Ker(G_i\to G_{i-1})$ and let $W$  be the submodule of $Z$ which is generated by all elements of $Z$ of degree $i+1$. Similarly we let  $Z'=\Ker(F_i^{\lin}\to F_{i-1}^{\lin})$ and $W'$ be the submodule of $Z'$ which is generated by all the elements of $Z'$ of degree $i+1$. Then the isomorphism  $Z\iso Z'$ induces an isomorphism $W\iso W'$.

Since by assumption
$H_i(\GG)_{i+1}=H_{i+1}(\GG)_{i+1}=0$, it follows that $G_{i+1}\to W$ is a minimal free presentation of $W$. By the same reason,  $F_{i+1}^{\lin}\to W'$ is a minimal free presentation of $W'$. This shows that the truncations of $\GG$ and $\FF^{\lin}$ are isomorphic up to homological degree $i+1$.
\end{proof}

We would like to remark that the linear strand $\FF^{\lin}$ of a module $M$ with initial degree $d$ is only determined by the submodule $N$ of $M$ which is generated by the  elements of degree $d$ of $M$. Moreover, $N=\Coker (F_1^{\lin}\to F_0^{\lin})$ if and only if $N$ is linearly presented.

Consider for example the ideal $I=(x^2,xy,y^3, z^2)\subset S=K[x,y,z]$ with minimal graded free resolution $\FF$.   Then $I$ and $J=(x^2,xy,z^2)$ have the same linear  strand. We have that $J\neq \Coker (F_1^{\lin}\to F_0^{\lin})$, since $J$ is not linearly presented.  Indeed, $\Coker (F_1^{\lin}\to F_0^{\lin})$ is isomorphic to $(x^2,xy)S\dirsum S(-2)$.

\begin{Corollary}
\label{who}
Let $\GG$ be a finite linear complex with initial degree $d$ satisfying the conditions {\em (b)} of Theorem~\ref{when}. Furthermore, let $W$ be a finitely generated graded $S$-module with minimal graded free resolution $\FF$ satisfying the property that $G_1\to G_0$ is isomorphic  to  $ F_1^{\lin}\to F_0^{\lin}$. Then $\GG$ is isomorphic to the linear strand of $W$.
\end{Corollary}

\begin{proof}
In the proof of Theorem~\ref{when} (b)\implies(a) we have seen that if $\GG$ is a finite linear complex with initial degree $d$ satisfying the conditions (b) of Theorem~\ref{when} and such that  $G_1\to G_0 \iso F_1^{\lin}\to F_0^{\lin}$, then $\GG\iso \FF^{\lin}$.
\end{proof}

\section{The generalized Eagon-Northcott complex}\label{Generalized Eagon-Northcott}

Let $F$ and $G$ be free $S$-modules of rank $m$ and $n$, respectively, with $m\leq n$, and let $\varphi\: G\to F$ be an $S$-module
homomorphism. We choose a basis $f_1,\ldots,f_m$ of $F$ and a basis $g_1,\ldots,g_n$ of $G$. Let $\varphi(g_j)=\sum_{i=1}^m\alpha_{ij}f_i$ for $j=1,\ldots, n$. The matrix $\alpha=(\alpha_{ij})$ describing $\varphi$ with respect to these bases is an $(m\times n)$-matrix with entries in $S$. The ideal of $m$-minors of this matrix is independent of the choice of these bases and is denoted $I_m(\varphi)$. It is know that with suitable grade conditions on $I_m(\varphi)$, the so-called Eagon-Northcott complex provides a free resolution of $I_m(\varphi)$, see \cite{EN}. There are nice and basis free descriptions of the Eagon-Northcott complex, see for example \cite{BV} and \cite{Ei}. For our purpose however a description of the Eagon-Northcott complex and its differentials in terms of natural bases is required. For the convenience of the reader and due to the lack of a suitable reference we recall this description.

We denote by $S(F)$ the symmetric algebra of $F$. Then $S(F)$ is isomorphic to the polynomial ring over $S$ in the variables $f_1,\ldots,f_m$. The module $G\tensor_S S(F)$ is a free $S(F)$-module with basis $g_1\tensor 1,\ldots,g_n\tensor 1$ and $\varphi$ gives rise to the $S(F)$-linear map
\[
G\tensor S(F)\to S(F) \text{ with $g_j\tensor 1\mapsto \varphi(g_j)$ for all $j$,}
\]
which in turn induces the Koszul complex
\[
K(\varphi)\: 0\to \bigwedge^n G\tensor S(F)\to \cdots \to \bigwedge^1 G\tensor S(F)\to \bigwedge^0 G\tensor S(F)\to 0,
\]
whose differential $\delta$ is defined as follows:
\begin{equation}
\label{partial}
\delta(g_{j_1}\wedge g_{j_2}\wedge \cdots \wedge g_{j_i}\tensor 1)=\sum_{k=1}^i(-1)^{k+1}g_{j_1}\wedge g_{j_2}\wedge \cdots \wedge g_{j_{k-1}}\wedge g_{j_{k+1}}\wedge\cdots \wedge  g_{j_i}\tensor \varphi(g_{j_k}).
\end{equation}
The Koszul complex $K(\varphi)$ splits into graded components
\[
K(\varphi)_i\: 0\to \bigwedge^i G\tensor S_0(F)\to \cdots \to \bigwedge^1 G\tensor S_{i-1}(F)\to \bigwedge^0 G\tensor S_i(F)\to 0,
\]
each of which is a complex of free $S$-modules.

We are interested in the $S$-dual of the $(n-m)$-th component of $K(\varphi)$ which is the complex
\begin{equation}
\label{dual}
0\to (\bigwedge^0 G\tensor S_{n-m}(F))^*\to \cdots \to (\bigwedge^{n-m-1} G\tensor S_{1}(F))^* \to (\bigwedge^{n-m} G\tensor S_{0}(F))^*\to 0.
\end{equation}
Let
\[
\mu: \bigwedge^{n-i} G\to \Hom_S(\bigwedge^i G, \bigwedge^n G)
\]
be the isomorphism  which assigns to $a\in \bigwedge^{n-i} G$ the $S$-linear map $\mu_a\in \Hom_S(\bigwedge^i G, \bigwedge^n G)$ with $\mu_a(b)=a\wedge b$. Composing $\mu$ with the isomorphism
\[
\Hom_S(\bigwedge^i G, \bigwedge^n G)\to (\bigwedge^i G)^*
\]
which is induced by the isomorphism $Sg_1\wedge\ldots \wedge g_n\iso S$ with $g_1\wedge\ldots \wedge g_n\mapsto 1$, we obtain the isomorphisms
\begin{eqnarray}
\label{star}
\bigwedge^{n-i} G\iso (\bigwedge^i G)^* \text{ for $i=0,\ldots,n$}.
\end{eqnarray}
Furthermore, if we use that for any two finitely generated free $S$-modules $A$ and $B$ there are natural isomorphisms
\begin{eqnarray}
\label{tensor}
(A\tensor B)^*\iso \Hom(A,B^*)\iso A^*\tensor B^*,
\end{eqnarray}
the complex  (\ref{dual}) becomes the complex
\[
\MC(\varphi)\: 0\to \bigwedge^n G\tensor S_{n-m}(F)^*\to \ldots \to \bigwedge^{m+1} G\tensor S_{1}(F)^* \to \bigwedge^{m} G\tensor S_{0}(F)^*\to 0,
\]
which is called the {\em Eagon-Northcott complex}.

There is a natural augmentation map
\[
\epsilon\:\; \bigwedge^{m} G\tensor S_0(F)^*\to I_m(\varphi)
\]
which assigns  to $g_{i_1}\wedge\cdots \wedge  g_{i_m}\tensor 1$ the  $m$-minor of $\alpha$  with the  columns $i_1,i_2,\ldots,i_m$. It is known that $\grade I_m(\varphi)\leq n-m+1$ and equality holds if and only if $\MC(\varphi)\to I_m(\varphi)\to 0$ is exact, that is, if and only if $\MC(\varphi)$ provides a free resolution of $I_m(\varphi)$, see \cite{BV}.

We now describe the differential of the Eagon-Northcott complex. The map
\[
\delta^*: (\bigwedge^{r} G\tensor S_{n-m-r}(F))^* \to (\bigwedge^{r+1} G\tensor S_{n-m-r-1}(F))^*
\]
is just the dual of the differential $\delta$ defined in (\ref{partial}). Thus  if we set $g_\sigma=g_{j_1}\wedge \cdots \wedge g_{j_r}$ for $\sigma=\{j_1<j_2<\cdots < j_r\}$ and set $f^{\bf{a}}=f_1^{a_1}f_2^{a_2}\cdots f_m^{a_m}$ with $|a|=a_1+\cdots +a_m=n-m-r$, then
\begin{eqnarray*}
\delta(g_\sigma\tensor f^{\bf{a}})& =& \sum_{k=1}^r(-1)^{k+1}g_{\sigma\setminus \{j_k\}}\tensor \varphi(g_{j_k})f^{\bf{a}}\\
&=& \sum_{k=1}^r \sum_{i=1}^m(-1)^{k+1}\alpha_{ij_k}(g_{\sigma\setminus \{j_k\}}\tensor f_1^{a_1}\cdots f_i^{a_i+1}\cdots f_m^{a_m}).
\end{eqnarray*}
Therefore,
\begin{equation}
\label{deltastar}
\delta^*((g_\sigma\tensor f^{\bf{a}})^*) = \sum_{j\in [n]\setminus \sigma}(-1)^{\sign(\sigma,j)} \sum_{i=1\atop a_i>0}^m\alpha_{ij}(g_{\sigma\union \{j\}}\tensor f_1^{a_1}\cdots f_i^{a_i-1}\cdots f_m^{a_m})^*,
\end{equation}
where
$
\sign(\sigma,j)=|\{i\in \sigma\: i<j\}|.
$

\medskip
Let $f^{(\ab)} =f_1^{(a_1)}\cdots f_m^{(a_m)}$ denote  the basis element in $(S_j(F))^*$ which is dual to $f^\ab=f_1^{a_1}\cdots f_m^{a_m}$. Furthermore, we set $f^{(\ab)} = 0$, if $a_i<0$ for some $i$.  Then by using the isomorphisms (\ref{star}) and (\ref{tensor}), we see that $(g_\sigma\tensor f^\ab)^*$ is mapped to $(-1)^{s(\sigma)}g_{\sigma^c}\tensor f^{(\ab)}$, where $\sigma^c=[n]\setminus \sigma$ and where $s(\sigma)$
is defined by the equation $$g_\sigma\wedge g_{\sigma^c}=(-1)^{s(\sigma)}g_{[n]}.$$
Thus, $\delta^*$ identifies with the map
\begin{equation}
\label{enpartial}
\partial(g_\sigma\tensor f^{(a)})=(-1)^{|\sigma^c|}\sum_{k=1}^{m+i}\sum_{i=1}^m(-1)^{k+1}\alpha_{ij_k} (g_{\sigma\setminus \{j_k\}}\tensor f_1^{(a_1)}\cdots f_i^{(a_i-1)}\cdots f_m^{(a_m)}),
\end{equation}
which is the differential of the Eagon-Northcott complex. In this formula, the sign $(-1)^{|\sigma^c|}$ which only depends on the homological degree of $g_\sigma\tensor f^{(a)}$, may be skipped, and we will do this in the final presentation of $\partial$.

Indeed, to see that (\ref{enpartial}) holds, we note that, due to  (\ref{deltastar}),
\begin{eqnarray*}
\label{implies}
&&\delta^*((-1)^{s(\sigma)} g_{\sigma^c}\tensor f^{(a)})\\
&=& \sum_{j\in  \sigma^c}(-1)^{\sign(\sigma,j)} \sum_{i=1\atop a_i>0}^m(-1)^{s(\sigma\union \{j\})}\alpha_{ij}(g_{\sigma^c\setminus \{j\}}\tensor f_1^{(a_1)}\cdots f_i^{(a_i-1)}\cdots f_m^{(a_m)}).\nonumber
\end{eqnarray*}
Thus
\begin{eqnarray*}
\label{againimplies}
&&\delta^*( g_{\sigma}\tensor f^{(a)})\\
&=& \sum_{j\in  \sigma} \sum_{i=1\atop a_i>0}^m(-1)^{s(\sigma^c)+\sign(\sigma^c,j)+s(\sigma^c\union \{j\})}\alpha_{ij}(g_{\sigma\setminus \{j\}}\tensor f_1^{(a_1)}\cdots f_l^{(a_i-1)}\cdots f_m^{(a_m)}).\nonumber
\end{eqnarray*}

Finally
\begin{eqnarray*}
(-1)^{s(\sigma^c\union\{j\})}g_{[n]}&=& g_{\sigma^c\union \{j\}}\wedge g_{\sigma\setminus\{j\}}=(-1)^{|\sigma^c|-\sign(\sigma^c,j)}g_{\sigma^c}\wedge g_j\wedge g_{\sigma\setminus\{j\}}\\
&=& (-1)^{|\sigma^c|-\sign(\sigma^c,j)+\sign(\sigma,j)}g_{\sigma^c}\wedge g_{\sigma}\\
&=&(-1)^{|\sigma^c|-\sign(\sigma^c,j)+\sign(\sigma,j)+s(\sigma^c)}g_{[n]}.
\end{eqnarray*}
It follows that $(-1)^{|\sigma^c|}(-1)^{\sign(\sigma,j)}=(-1)^{s(\sigma^c)+\sign(\sigma^c,j)+s(\sigma^c\union\{j\})}$, as desired.

\medskip
In order to simplify notation we set $\bb(\sigma;\ab)= g_\sigma\tensor f^{(\ab)}$. Then the elements $\bb(\sigma;\ab)$ with $|\sigma|=m+i$  and $\ab=(a_1,\ldots,a_m)$ such that  $a_1+\cdots +a_m=i$ form a basis of $\bigwedge ^{m+i}G\tensor S_i(F)^*$, and
\[
\partial(\bb(\sigma;\ab))=\sum_{k=1}^{m+i}\sum_{\ell=1}^m(-1)^{k+1}\alpha_{\ell j_k}\bb(\sigma\setminus \{j_k\};\ab-\eb_{\ell}).
\]
Here $\eb_1,\ldots, \eb_m$ is the canonical basis of $\ZZ^m$.

\begin{Corollary}
\label{en}
Suppose that $\grade I_m(\varphi)=n-m+1$. Then $I_m(\varphi)$ has a linear resolution and
\[
\beta_i(I_m(\varphi))={n\choose m+i}{m+i-1\choose m-1} \quad \text{for}\quad i=0,\ldots,n-m.
\]
\end{Corollary}

\medskip
Now let $\Delta$ be a simplicial complex on the vertex set $[n]$. We set $\MC_i(\varphi)= \bigwedge^{m+i}G\tensor {S_i(F)}^{*}$ and denote $\MC_i(\Delta;\varphi)$ the free submodule of $\MC_i(\varphi)$ generated by all $\bb(\sigma;\ab)$ such that $\sigma\in \Delta$ with $|\sigma|=m+i$, and $\ab\in \ZZ^m_{\geq 0}$ with $|\ab|=i$.
Since $\partial(\bb(\sigma;\ab))\in \MC_{i-1}(\Delta;\varphi)$ for all  $\bb(\sigma;\ab)\in \MC_i(\Delta;\varphi)$, we obtain the subcomplex
\[
\MC(\Delta;\varphi)\: 0\to \MC_{n-m}(\Delta;\varphi)\to \cdots  \to\MC_1(\Delta;\varphi)\to  \MC_0(\Delta;\varphi)\to 0
\]
of $\MC(\varphi)$ which we call the {\em generalized Eagon-Northcott complex} attached to the simplicial complex $\Delta$ and the module homomorphism $\varphi: G\to F$. In particular, if $\Delta$ is just a simplex on $n$ vertices, then $\MC(\Delta;\varphi)$ coincides with $\MC(\varphi)$.

Moreover, note that the generalized Eagon-Northcott complex of a simplicial complex $\Delta$ is determined by those facets of $\Delta$ whose dimension is at least $m-1$. More precisely, if we obtain the simplicial complex $\Delta'$ from $\Delta$ by removing all facets of dimension less than $m-1$, then the complexes $\MC(\Delta;\varphi)$ and $\MC(\Delta';\varphi)$ are the same.

\section{The generalized Eagon-Northcott complex as a linear strand of a module}\label{generalized as linear strand}

Let $X$ be an $(m\times n)$-matrix of indeterminates $x_{ij}$. We fix a field $K$ and let $S$ be the polynomial ring over $K$ in the variables $x_{ij}$. Moreover, let $\varphi\: G\to F$ be the $S$-module homomorphism of free $S$-modules given by the matrix $X$. In the rest of this paper, we fix this situation.

Now we give a $(\ZZ^m\times\ZZ^n)$-grading to the polynomial ring $S$, by setting $\mathrm{mdeg}(x_{ij})=(e_i,\varepsilon_j)$ where $e_i$ is the $i$-th canonical basis vector of $\ZZ^m$ and $\varepsilon_j$ is the $j$-th canonical basis vector of $\ZZ^n$. Now, let $\Delta$ be a simplicial complex. Then the chain complex $\MC(\Delta;\varphi)$ inherits this grading. More precisely, for each $i$, the degree of a basis element $\bb(\sigma;{\bf{a}})$ of $\MC_i(\Delta;\varphi)$ with $\sigma=\{j_1,\ldots,j_{m+i}\}$ is set to be $({\bf{a}}+{\bf{1}},\gamma)\in \ZZ^m\times\ZZ^n$, where $\gamma=\varepsilon_{j_1}+\cdots+\varepsilon_{j_{m+i}}$, and $\bf{1}$ is the vector in $\ZZ^m$ whose entries are all equal to $1$. Then, one can see that all the homomorphisms in the chain complex $\MC(\Delta;\varphi)$ is homogeneous.

To characterize all simplicial complexes for which the generalized Eagon-Northcott complex $\MC(\Delta;\varphi)$ is the linear strand of a finitely generated graded $S$-module, we need to introduce a concept about simplicial complexes which is crucial in the sequel. Let $\Delta$ be a simplicial complex on the vertex set $V$, and let $\sigma$ be a  subset of $V$. Then recall that $\sigma$ is called a \emph{minimal nonface} of $\Delta$ if $\sigma\notin \Delta$, but all its proper subsets belong to $\Delta$.

\begin{Theorem}\label{missing}
Let $\Delta$ be a simplicial complex on $n$ vertices, and let $m$ be a positive integer. Then the following conditions are equivalent:
\begin{enumerate}
\item[{\em (a)}] $\MC(\Delta;\varphi)$ is the linear strand of a finitely generated graded $S$-module with initial degree~$m$.
\item[{\em (b)}] $\Delta$ has no minimal nonfaces of cardinality $\geq m+2$.
\end{enumerate}
\end{Theorem}

\begin{proof}
We may assume that $m\leq \dim \Delta +1$, because otherwise $\MC(\Delta;\varphi)=0$ and the assertion is trivial. By Theorem~\ref{when} it suffices to  show that $\Delta$ has no minimal nonfaces of cardinality $\geq  m+2$ if and only if $H_i(\MC(\Delta;\varphi))_{i+m+j}=0$ for all $i>0$ and for $j=0,1$.

First suppose that $j=0$, and let $z$ be a nonzero cycle of degree $m+i$ in $\MC_i(\Delta;\varphi)$ for some $i>0$. Then $z$ is also a cycle in $\MC_i(\varphi)$. Since $\MC(\varphi)$ is exact, $z$ is a boundary in $\MC_i(\varphi)$. So there is a nonzero element $f$ of degree $m+i$ in $\MC_{i+1}(\varphi)$ such that $\partial(f)=z$ which is a contradiction, since ${\MC_{i+1}(\varphi)}_{m+i}=0$. This implies that there is no nonzero cycle of degree $m+i$ in $\MC_i(\Delta;\varphi)$, and hence $H_i(\MC(\Delta;\varphi))_{i+m}=0$.

Therefore, to prove the theorem it is enough to show that $\Delta$ has no minimal nonfaces of cardinality $\geq m+2$ if and only if $H_i(\MC(\Delta;\varphi))_{i+m+1}=0$ for all $i>0$. Equivalently, we show that $\Delta$ has a minimal nonface of cardinality at least $m+2$ if and only if $H_i(\MC(\Delta;\varphi))_{i+m+1}\neq 0$ for some $i>0$.

To prove this, suppose that there is a minimal nonface $\sigma=\{j_1,\ldots,j_{m+1+i}\}$ of $\Delta$ for some $i>0$. Therefore, $\sigma\notin \Delta$ and for each $t=1,\ldots,m+1+i$, we have $\sigma\setminus \{j_t\}\in \Delta$. Now let $z=\partial(\bb(\sigma;{\bf{a}}))$ for some ${\bf{a}}\in \ZZ^m_{\geq 0}$ with $|{\bf{a}}|=i+1$. Then $z\in \MC_i(\Delta;\varphi)$, and $z$ is a nonzero element of $\MC_i(\varphi)$. Since $z$ is a cycle in $\MC_i(\varphi)$, it is also a cycle in $\MC_i(\Delta;\varphi)$. On the other hand, the multidegree of $z$ is $(\bf{a}+\bf{1},\gamma)$ where $\gamma=\varepsilon_{j_1}+\cdots+\varepsilon_{j_{m+1+i}}$. Since $\sigma\notin \Delta$, there is no basis element of multidegree $(\bf{a}+\bf{1},\gamma)$ in $\MC_{i+1}(\Delta;\varphi)$, and hence $z$ is not a boundary in $\MC_i(\Delta;\varphi)$. This implies that $H_i(\MC(\Delta;\varphi))_{i+m+1}\neq 0$.

Conversely, suppose that $H_i(\MC(\Delta;\varphi))_{i+m+1}\neq 0$ for some $i>0$. Then there is a cycle $z\in \MC_i(\Delta;\varphi)$ which is not a boundary. We may assume that $z$ is multihomogeneous with $\mathrm{mdeg}(z)=(\bf{a}+\bf{1},\gamma)$ with $\gamma=\varepsilon_{j_1}+\cdots+\varepsilon_{j_{m+1+i}}$ and $|{\bf{a}}|=i+1$. Since $z$ is also a cycle in $\MC_i(\varphi)$ and since $\MC(\varphi)$ is exact, we deduce that $z$ is a boundary in $\MC_i(\varphi)$. Thus there is a nonzero multihomogeneous element $f\in \MC_{i+1}(\varphi)$ such that $\partial(f)=z$. Since $\MC(\varphi)$ is a multigraded resolution, it follows that $\mathrm{mdeg}(f)=(\bf{a}+\bf{1},\gamma)$. Since any two basis elements of $\MC_{i+1}(\varphi)$ have different multidegrees, it follows that $f=\lambda\bb(\sigma;{\bf{a}})$ for some $0\neq \lambda\in K$ with $\sigma=\{j_1,\ldots,j_{m+1+i}\}$. Because $z$ is not a boundary in $\MC_i(\Delta;\varphi)$, the basis element $\bb(\sigma;{\bf{a}})$ does not belong to $\MC_{i+1}(\Delta;\varphi)$, and hence $\sigma\notin \Delta$. On the other hand, we have
\[
z=\partial(f)=\lambda\sum_{k=1}^{m+1+i}\sum_{\ell=1}^m{(-1)}^{k+1}x_{\ell j_{k}}\bb(\sigma\setminus \{j_k\};{\bf{a}}-{\bf{e}}_{\ell}).
\]
Since the basis elements of $\MC_{i+1}(\Delta;\varphi)$ form a subset of the basis elements of $\MC_{i+1}(\varphi)$, it follows from this presentation of $z$ that $\bb(\sigma\setminus \{j_k\};{\bf{a}}-{\bf{e}}_{\ell})\in \MC_{i}(\Delta;\varphi)$ for all $j_k$ and $\ell$ for which $\bb(\sigma\setminus \{j_k\};{\bf{a}}-{\bf{e}}_{\ell})\neq 0$. Since $z$ is not zero, there exist $j_k$ and $l$ with $\bb(\sigma\setminus \{j_k\};{\bf{a}}-{\bf{e}}_{\ell})\neq 0$. Therefore, ${\bf{a}}-{\bf{e}}_{\ell}\in \ZZ^m_{\geq 0}$. Hence $\bb(\sigma\setminus \{j_t\};{\bf{a}}-{\bf{e}}_{\ell})\neq 0$ for all $t=1,\ldots,m+1+i$. It follows that $\bb(\sigma\setminus \{j_t\};{\bf{a}}-{\bf{e}}_{\ell})\in \MC_{i+1}(\Delta;\varphi)$ for all $t=1,\ldots,m+1+i$. Therefore, $\sigma\setminus \{j_t\}$ is a face of $\Delta$ for all $t=1,\ldots,m+1+i$, so that $\sigma$ is a minimal nonface of $\Delta$ of cardinality $m+1+i\geq m+2$, as desired.
\end{proof}

We would like to remark that a subclass of simplicial complexes satisfying condition (b) of Theorem~\ref{missing} has been considered in \cite{Ne} for a different purpose. There, minimal nonfaces are called {\em missing faces}.

Now, we give some examples of other classes of simplicial complexes which satisfy condition (b) of Theorem~\ref{missing}. First recall that a \emph{flag} simplicial complex is a simplicial complex whose minimal nonfaces all have cardinality equal to two. Flag complexes are exactly clique complexes of graphs, i.e. a simplicial complex whose faces are cliques of a graph. In \cite{KN}, a generalization of flag complexes has been introduced which we recall it in the following.

In \cite{KN} a subset $T$ of $V$ is called a \emph{critical clique} of $\Delta$ if each pair of vertices form a $1$-dimensional face of $\Delta$ and $T\setminus \{v\}\in \Delta$ for some $v\in T$. Let $\dim \Delta=d-1$ and let $i$ be an integer with $1\leq i\leq d$. Then $\Delta$ is called $i$-\emph{banner} if  every critical clique $T$ of cardinality at least $i+1$ is a face of $\Delta$. In \cite[Lemma~3.3]{KN} it is shown that a simplicial complex is $1$-banner if and only if it is $2$-banner, and this is equivalent to being flag.

Now, let $1\leq m\leq d$. Then an $(m+1)$-banner satisfies condition (b) of Theorem~\ref{missing}, namely has no minimal nonfaces of cardinality $\geq m+2$, see also \cite[Lemma~3.6]{KN}. In particular, all flag complexes satisfy condition (b) of Theorem~\ref{missing}, because $i$-banner implies $(i+1)$-banner for all $i$.

We also would like to remark that not all simplicial complexes with no minimal nonfaces of cardinality $\geq m+2$ are $(m+1)$-banner. For example, let $\Delta$ be the simplicial complex with the vertex set $\{1,2,3,4\}$ whose facets are $\{1,2,3\}$, $\{1,3,4\}$ and $\{2,3,4\}$ and let $m=2$. Then $\Delta$ has no minimal nonfaces of cardinality $4$, but it is not $3$-banner. Indeed, $T=\{1,2,3,4\}\notin \Delta$ but it is a critical clique of $\Delta$.

\section{The linear strand of determinantal facet ideals and binomial edge ideals}\label{determinantal}

A \emph{clutter} $C$ on the vertex set $[n]$ is a collection of subsets of $[n]$ such that there is no containment between its elements. An element of $C$ is called a \emph{circuit}. In this paper, we assume that each vertex of a clutter $C$ belongs to some circuits of $C$. If all circuits of $C$ have the same cardinality $d$, then $C$ is said to be an \emph{$d$-uniform} clutter.

Let $C$ be an $m$-uniform clutter. To each circuit $\tau\in C$ with $\tau=\{j_1,\ldots,j_m\}$ and $1\leq j_1<j_2<\cdots <j_m\leq n$ we assign the $m$-minor ${\bf{m}}_\tau$ of $X$ which is determined  by the columns $1\leq j_1<j_2<\cdots <j_m\leq n$. We denote by  $J_C$ the ideal in $S$ which is generated by the minors ${\bf{m}}_\tau$ with $\tau\in C$. This ideal which has first been considered in \cite{EHHM} is called the {\em determinantal facet ideal} of $C$ since the circuits of $C$ may be considered as the facets of a simplicial complex. In the case that $C$ is a $2$-uniform clutter, $C$  may be viewed as  a graph $G$, and in this case $J_\MC=J_G$, where $J_G$ is the binomial edge ideal of $G$, as defined in \cite{HHHKR}.

A {\em clique} of $C$ is a subset $\sigma$ of $[n]$ such that each $m$-subset $\tau$ of $\sigma$ is a circuit of $C$. We denote by $\Delta(C)$ the simplicial complex whose faces are the cliques of $C$ which is called the {\em clique complex} of $C$. A clutter is called {\em complete} if its clique complex is a simplex.

Note that the clique complex of an $m$-clutter has no minimal nonface of cardinality $\geq m+2$. The example at the end of Section~\ref{generalized as linear strand} also shows that not all simplicial complexes which satisfy condition (b) of Theorem~\ref{missing} are clique complexes of clutters. Here we also give an example of a $3$-clutter whose clique complex is not $4$-banner, although all clique complexes of $2$-clutters are $3$-banner, as mentioned before. Let $C$ be the $3$-clutter on the vertices $\{1,\ldots,6\}$ with the circuits $\{1,2,3\}$, $\{1,3,4\}$, $\{2,3,4\}$, $\{3,4,5\}$, $\{3,4,6\}$, $\{3,5,6\}$ and $\{4,5,6\}$. Then $\Delta(C)=\langle \{1,2,3\}, \{1,3,4\}, \{2,3,4\}, \{3,4,5,6\}\rangle$. Thus $\Delta(C)$ is not $4$-banner, since $\{1,2,3,4\}$ is a critical clique of $\Delta(C)$ which is not a face.

\medskip
The following theorem is the main result of this section.

\begin{Theorem}
\label{hope}
Let $C$ be an $m$-uniform clutter on the vertex set $[n]$, and let $\FF$ be the minimal graded free resolution of $J_C$. Then
\[
\FF^{\lin}\iso \MC(\Delta(C);\varphi).
\]
\end{Theorem}

\begin{proof}
To prove the theorem, by Corollary~\ref{who} and Theorem~\ref{missing}, it suffices to prove the following
\[
F_1^{\lin}\to F_0^{\lin}\iso \MC_1(\Delta(C);\varphi) \to \MC_0(\Delta(C);\varphi).
\]

First, note that $J_C$ is a homogeneous ideal of $S$ with respect to the $(\ZZ^m\times \ZZ^n)$-multigrading introduced in Section~\ref{generalized as linear strand}. Now, we define
\[
{\MC}_0(\Delta(C);\varphi)\stackrel{\psi}{\longrightarrow} J_C
\]
with $\psi(\bb(\sigma;{\bf{0}}))={\bf{m}}_{\sigma}$ for each $\sigma=\{j_1,\ldots,j_m\}\in C$, so that it is a minimal free presentation of $J_C$. Since $\MC(\varphi)$ is the resolution of $I_m(\varphi)$, the following commutative diagram implies that $\psi \circ \partial =0$.

\begin{displaymath}\label{diagram}
\xymatrix{
\MC_1(\Delta(C);\varphi) \ar[r]^{\partial} \ar@{^{(}->}[d] &
\MC_0(\Delta(C);\varphi) \ar[r]^{\psi} \ar@{^{(}->}[d] &
J_C \ar@{^{(}->}[d] \\
\MC_1(\varphi) \ar[r]^{\partial} & \MC_0(\varphi) \ar[r] & I_m(\varphi)}
\end{displaymath}

Let $Z=\ker \psi$ and $Z^{\lin}$ be the $S$-module generated by the elements of $Z$ whose degree is $m+1$. Since $\MC_1(\Delta(C);\varphi)$ is generated in degree $m+1$, it follows that $\partial$ induces a homogeneous homomorphism $\MC_1(\Delta(C);\varphi)\stackrel{\partial'}\longrightarrow Z^{\lin}$. The desired result follows once we have shown that $\partial'$ induces an isomorphism of vector spaces ${\MC_1(\Delta(C);\varphi)}_{m+1}\stackrel{\partial'}\longrightarrow {(Z^{\lin})}_{m+1}=Z_{m+1}$.

Let $z$ be a nonzero element of $Z^{\lin}$ of multidegree $({\bf{e}}_s+\bf{1},\gamma)$ for some $s=1,\ldots,m$ where $|\gamma|=m+1$. Let $\gamma=\varepsilon_{j_1}+\cdots+\varepsilon_{j_{m+1}}$. Since $z\neq 0$, at least $m$ vectors among $\varepsilon_{j_1},\ldots,\varepsilon_{j_{m+1}}$ are distinct, say $\varepsilon_{j_1},\ldots,\varepsilon_{j_{m}}$, and the basis element $\bb(\sigma;\bf{0})$ appears in the presentation of $z$ as the linear combination of the basis elements of $\MC_0(\Delta(C);\varphi)$, where $\sigma=\{j_1,\ldots,j_{m}\}$.

First suppose that $j_{m+1}\in \sigma$. Then according to the $(\ZZ^m\times \ZZ^n)$-multigrading, the only basis element of ${\MC}_0(\Delta(C),\varphi)$ which contributes to this multidegree in $z$, is $\bb(\sigma;{\bf{0}})$. Then, it follows that $z=\lambda x_{sj_{m+1}}\bb(\sigma;{\bf{0}})$ for some $0\neq \lambda\in K$, and hence we have $\psi(z)\neq 0$, a contradiction. Therefore, $j_{m+1}\notin \sigma$. Let $\tau=\{j_1,\ldots,j_{m+1}\}$, and without loss of generality assume that $j_1<\cdots<j_{m+1}$. Then all possible basis elements of ${\MC}_0(\Delta(C),\varphi)$ which could contribute to this multidegree in $z$ are of the form $\bb(\tau\setminus \{j_k\};{\bf{0}})$. By comparing the multidegrees, we observe that $z=\sum_{k=1}^{m+1}\lambda_kx_{sj_k}\bb(\tau\setminus \{j_k\};{\bf{0}})$ for some $\lambda_k\in K$. From the above diagram it follows that $z$ is a boundary of $\MC(\varphi)$. This implies that there exists a nonzero multihomogeneous element $f\in {\MC}_1(\Delta;\varphi)$ with $\partial(f)=z$. Since $\mathrm{mdeg}(f)=\mathrm{mdeg}(z)$, it follows that $f=\lambda {\bf{b}}(\tau;{\bf{e}}_s)$ for some $0\neq \lambda\in K$. Therefore, $z=\partial(f)=\lambda \sum_{k=1}^{m+1}{(-1)}^{k+1}x_{sj_k}\bb(\tau\setminus \{j_k\};{\bf{0}})$. This implies that for each $k=1,\ldots,m+1$, we have $\lambda_k={(-1)}^{k+1}\lambda$ and the basis element $\bb(\tau\setminus \{j_k\};{\bf{0}})$ appears in the presentation of $z$ with a nonzero coefficient. Hence for each $k=1,\ldots,m+1$, $\tau\setminus \{j_k\}$ is a circuit of $C$ which implies that $\tau$ is a clique of $C$. This shows that the set $\{r_s(\tau):1\leq s\leq m, \tau\in \Delta(C), |\tau|=m+1\}$ with $r_s(\tau)=\sum_{k=1}^{m+1}{(-1)}^{k+1}x_{sj_k}\bb(\tau\setminus \{j_k\};{\bf{0}})$ for each clique $\tau=\{j_1,\ldots,j_{m+1}\}$ of $C$, is a system of generators of $Z^{\mathrm{lin}}$. Since $\mathrm{mdeg}(r_s(\tau))\neq \mathrm{mdeg}(r_{s'}(\tau'))$ if $(s,\tau)\neq (s',\tau')$, we conclude that these elements form a basis of $Z_{m+1}$. Moreover, each basis element ${\bf{b}}(\tau;e_s)$ is mapped by $\partial'$ to $r_s(\tau)$. Thus, we get the desired result.
\end{proof}

Recall that for a simplicial complex $\Delta$ and an integer $i$, the $i$-th skeleton $\Delta^{(i)}$ of $\Delta$ is the subcomplex of $\Delta$ whose faces are those faces of $\Delta$ whose dimension is at most $i$.

\begin{Corollary}\label{skeleton}
Let $\Delta$ be a simplicial complex whose facets all have dimension at least $m-1$, and let $C$ be the $m$-clutter whose circuits are exactly the facets of $\Delta^{(m-1)}$. Then $\MC(\Delta;\varphi)$ is the linear strand of $J_C$ if and only if $\Delta=\Delta(C)$.
\end{Corollary}

\begin{proof}
It is enough to note that by our assumptions we have $\Delta \subseteq \Delta(C)$, so that by Theorem~\ref{hope} the conclusion follows.
\end{proof}

Recall that the $f$-vector $(f_0(\Delta),f_1(\Delta),\ldots)$ of a simplicial complex $\Delta$ is the integer vector such that $f_t(\Delta)=|\{\sigma\in \Delta:\dim \sigma=t\}|$ for each $t=0,\ldots,\dim \Delta$. The following corollary is a consequence of Theorem~\ref{hope}. In particular, in the case $m=2$, it proves the conjecture in \cite[page~338]{SK}.

\begin{Corollary}\label{linearbetti}
Let $C$ be an $m$-uniform clutter. Then
\[
\beta_{i,i+m}(J_C)={m+i-1\choose m-1} f_{m+i-1}(\Delta(C)),
\]
for all $i$.
\end{Corollary}

Another straightforward consequence of Theorem~\ref{hope} is the following result.

\begin{Corollary}\label{projdim}
Let $C$ be an $m$-uniform clutter. Then the length of the linear strand of $J_C$ is equal to $\dim \Delta(C)-m+1$. In particular, $\projdim J_C \geq \dim \Delta(C)-m+1$.
\end{Corollary}

The projective dimension of  the determinantal facet ideals in general is bigger than the lower bound given in Corollary~\ref{projdim}. This could be the case, even for $2$-uniform clutters, namely binomial edge ideals of graphs. For example, it follows from \cite[Theorem~1.1]{EHH} that for a block graph $G$ with $n$ vertices, $\projdim J_G=n-c-1$ where $c$ is the number of connected components of $G$. One can see that if $G$ is a non-complete block graph $G$ over $n$ vertices, $n-c-1>\dim \Delta(G)-1$, and hence in this case the lower bound given in Corollary~\ref{projdim} is strictly less than the projective dimension.

\section{Determinantal facet ideals with linear resolution}\label{linear resolution}

In this section, we determine when the minimal graded free resolution of determinantal facet ideals coincide with their linear strand. In other words, we characterize all $m$-uniform clutters whose determinantal facet ideal has a linear resolution. Moreover, we show that having a linear resolution is equivalent to being linearly presented for such ideals. Indeed, applying Theorem~\ref{hope}, we show that these properties occur for $J_C$ if and only if $C$ is a complete clutter.

Recall that a graded ideal $I$ in $S$ which is generated in degree $d$ is called \emph{linearly presented} if $\beta_{1,j}(I)=0$ for all $j\neq d+1$. Moreover, $I$ is said to have a \emph{linear resolution} if $\beta_{i,j}(I)=0$ for all $i$ and $j\neq i+d$.

The following theorem is the main result of this section which in particular recovers the case of binomial edge ideals from \cite{KS}. Here, we consider $S$ as a $\ZZ^n$-graded ring by setting $\deg(x_{ij})=\varepsilon_j$ for all $i=1,\ldots,m$, where by $\varepsilon_j$ we mean the $j$-th canonical basis vector of $\ZZ^n$.

\begin{Theorem}\label{linear res}
Let $C$ be an $m$-uniform clutter. Then the following conditions are equivalent:
\begin{enumerate}
\item[{\em (a)}] $J_C$ has a linear resolution.
\item[{\em (b)}] $J_C$ is linearly presented.
\item[{\em (c)}] $C$ is a complete clutter.
\end{enumerate}
\end{Theorem}

To prove the above theorem, we need the following lemma.

\begin{Lemma}\label{linear lemma}
Let $C$ be an $m$-uniform clutter on $[n]$ and suppose that $J_C$ is linearly presented. Then the following conditions hold:
\begin{enumerate}
\item[{\em (a)}] for each $\sigma_1, \sigma_2\in \Delta(C)$ with $|\sigma_1\cap \sigma_2|\geq m-1$, we have $\sigma_1\cup \sigma_2\in \Delta(C)$.
\item[{\em (b)}] for circuits $\tau$ and $\sigma$ of $C$, $\rho\in \Delta(C)$ with $\sigma,\tau\subseteq \rho$, and  $c\notin \rho$, it follows that $\tau\cup \{c\}\in \Delta(C)$ if and only if $\sigma\cup \{c\}\in \Delta(C)$.
\end{enumerate}
\end{Lemma}

\begin{proof}
(a) Before proving (a) we first show:

($\ast$) for each $(m+1)$-subset $\sigma$ of $[n]$ which contains at least two circuits of $C$, we have $\sigma\in \Delta(C)$.

Suppose on the contrary this is not the case, and suppose $\sigma=\{a_1,\ldots,a_{m+1}\}$ is an $(m+1)$-subset of $[n]$ which contains at least two circuits of $C$, but $\sigma\notin \Delta(C)$. We may assume that $\tau_1=\{a_1,\ldots,a_m\}$ and $\tau_2=\{a_2,\ldots,a_{m+1}\}$ are circuits of $C$. Let $(\FF,\delta)$ be the minimal graded free resolution of $J_C$. Note that $F_0=\MC_0(\Delta(C);\varphi)$ and $F_1^{\mathrm{lin}}=\MC_1(\Delta(C);\varphi)$ by Theorem~\ref{hope}. So, $F_1=\MC_1(\Delta(C);\varphi)\dirsum F'_1$, for some free graded $S$-module $F'_1$. Let $z={\bf{m}}_{\tau_1}{\bf{b}}(\tau_2;{\bf{0}})-{\bf{m}}_{\tau_2}{\bf{b}}(\tau_1;{\bf{0}})\in \MC_0(\Delta(C);\varphi)$. Since $z\in \ker \psi$, there exists a nonzero homogeneous element $f=(f_1,f_2)\in F_1$ where $f_1\in \MC_1(\Delta(C);\varphi)$, $f_2\in F'_1$ and $\delta (f)=z$. On the other hand, $\deg(z)=\varepsilon_{a_1}+2\varepsilon_{a_2}+\cdots+2\varepsilon_{a_m}+\varepsilon_{a_{m+1}}$, so that $\deg(f_1)=\deg(f_2)=\varepsilon_{a_1}+2\varepsilon_{a_2}+\cdots+2\varepsilon_{a_m}+\varepsilon_{a_{m+1}}$, because $\delta$ is homogeneous. Any basis element of $\MC_1(\Delta(C);\varphi)$ appearing in the presentation of $f_1$ is of the form ${\bf{b}}(\sigma;e_t)$ for some $t=1,\ldots,m$. Since by the assumption $\sigma$ is not a clique in $C$, no such basis element exists in $\MC_1(\Delta(C);\varphi)$, and hence $f_1=0$. Therefore, we have $f_2\neq 0$, since $f\neq 0$. This contradicts to our assumption that $J_{\MC}$ is linearly presented.

Now we show that ($\ast$) implies (a). Suppose that condition ($\ast$) holds. Let $\sigma_1=\{v_1,\ldots,v_k,u_1,\ldots,u_r\}$ and $\sigma_2=\{v_1,\ldots,v_k,w_1,\ldots,w_s\}$ be two cliques of $C$ with $|\sigma_1\cap \sigma_2|=k\geq m-1$. We show that $\sigma_1\cup \sigma_2$ is a clique. For this, it is enough to show that each $m$-subset $H=\{v_{a_1},\ldots,v_{a_{t}},u_{b_1},\ldots,u_{b_{p}},w_{c_1},\ldots,w_{c_{q}}\}$ of $\sigma_1\cup \sigma_2$ in which $t+p+q=m$ and $p,q\geq 1$, is a circuit of $C$. We prove this by induction on $p+q$. First assume that $p=q=1$. Then we have $\tau_1=\{v_{a_1},\ldots,v_{a_{m-1}},u_{b_1}\}\in C$ and $\tau_2=\{v_{a_1},\ldots,v_{a_{m-1}},w_{c_1}\}\in C$,
since $k\geq m-1$ and $\sigma_1,\sigma_2\in \Delta(C)$. Thus, $\{v_{a_1},\ldots,v_{a_{m-1}},u_{b_1},w_{c_1}\}$ contains at least two circuits $\tau_1$ and $\tau_2$ of $C$, and hence by ($\ast$) it is a clique in $C$ on $m+1$ vertices. Therefore, all its $m$-subsets are circuits of $C$, and in particular, $\{v_{a_1},\ldots,v_{a_{m-2}},u_{b_1},w_{c_1}\}\in C$, so that we are done in this case.
Now suppose that $p+q=\ell>2$. Then by induction hypothesis, we have $(H\cup \{v_{a_{t+1}}\})\setminus \{u_{b_p}\}\in C$ and $(H\cup \{v_{a_{t+1}}\})\setminus \{w_{c_q}\}\in C$ for all $a_{t+1}\in [k]\setminus \{a_1,\ldots,a_t\}$. Therefore, by ($\ast$) $H\cup \{v_{a_{t+1}}\}\in \Delta(C)$ which implies that $H$ is a circuit of $C$, as desired.

(b) It is enough to show that $\sigma\cup \{c\}\in \Delta(C)$ if $\tau\cup \{c\}\in \Delta(C)$. If $\tau=\sigma$, there is nothing to prove. Suppose $\tau\cup \{c\}\in \Delta(C)$, $a\in \sigma\setminus \tau$ and $b\in \tau\setminus \sigma$. Then $(\tau\setminus \{b\})\cup \{c\}$ is a circuit of $C$. Since $\rho\in \Delta(C)$, we have $\tau'=(\tau\setminus \{b\})\cup \{a\}$ is also a circuit of $C$. Then it follows by (a) that $\tau'\cup \{c\}\in \Delta(C)$. Since $|\sigma\setminus \tau'|<|\sigma\setminus \tau|$, induction on $|\sigma\setminus \tau|$ yields the result.
\end{proof}

Now we are ready to prove Theorem~\ref{linear res}.

\begin{proof} [Proof of Theorem~\ref{linear res}]
(a)\implies (b) is clear, and (c)\implies (a) follows since the Eagon-Northcott complex is the minimal graded free resolution of the determinantal facet ideal of a complete clutter.

(b)\implies (c): We may assume that $C$ is a clutter on the vertex set $[n]$. Assume that $J_C$ is linearly presented, and let $\FF$ be the minimal graded free resolution of $J_C$. Then by Theorem~\ref{hope} we have $F_1=\MC_1(\Delta(C);\varphi)$.
Now suppose on the contrary that $C$ is not complete. We claim that there exist two circuits in $C$ whose union is not a clique in $C$. Now we prove the claim. Since $C$ is not complete, there is an $m$-subset $A=\{c_1,\ldots,c_m\}$ of $[n]$ which is not a circuit of $C$. Let $k$ be the biggest integer such that $\{c_1,\ldots,c_k\}$ is contained in a circuit $\sigma_1$. Then $k<m$, because $A$ is not a circuit. On the other hand, $c_{k+1}\in \sigma_2$ for some $\sigma_2\in C$, but $c_{k+1}\notin \sigma_1$ and therefore $\sigma_2\neq \sigma_1$. Moreover, $\sigma_1\cup \sigma_2$ is not a clique in $C$, since otherwise the set $\{c_1,\ldots,c_k,c_{k+1}\}$ is contained in a circuit of $C$ which is a contradiction, because $k$ is the biggest integer with this property. This proves the claim.

Let $z={\bf{m}}_{\sigma_2}{\bf{b}}(\sigma_1;{\bf{0}})-{\bf{m}}_{\sigma_1}{\bf{b}}(\sigma_2;{\bf{0}})$. Since $z\in \ker \psi$, there exists a nonzero multihomogeneous element $w\in \MC_1(\Delta(C);\varphi)$ of the same multidegree as $z$ with $\partial(w)=z$. Let $A=\{a\in \sigma_1\setminus \sigma_2 : \sigma_2\cup \{a\}\in \Delta(C)\}$ and $B=\{b\in \sigma_2\setminus \sigma_1 : \sigma_1\cup \{b\}\in \Delta(C)\}$. Since $\partial(w)=z$, it follows that $A,B\neq \emptyset$. Now let $\rho_1=\sigma_1\cup B$ and $\rho_2=\sigma_2\cup A$. Using repeatedly part~(a) of Lemma~\ref{linear lemma}, it follows that $\rho_1,\rho_2\in \Delta(C)$. Since $\rho_1\cup \rho_2=\sigma_1\cup \sigma_2$, we deduce that $\rho_1\cup \rho_2\notin \Delta(C)$. Then by Lemma~\ref{linear lemma} part~(a), we have $|\rho_1\cap \rho_2|<m-1$. So, for $i=1,2$ we choose $\tau_i$ such that $\rho_1\cap \rho_2\subseteq \tau_i\subseteq \rho_i$ and $|\tau_i|=m$. In particular, $\tau_1\neq \tau_2$. Now consider $z'={\bf{m}}_{\tau_2}\bb(\tau_1;{\bf{0}})-{\bf{m}}_{\tau_1}\bb(\tau_2;{\bf{0}})$ which is a nonzero element of $\ker \psi$.

We claim that $\tau_1\cup \{c\}\notin \Delta(C)$ for any $c\in \tau_2\setminus \tau_1$. Then it follows that there is no nonzero multihomogeneous element in $\MC_1(\Delta(C);\varphi)$ of the same multidegree as $z'$, so that  $z'\notin \partial (\MC_1(\Delta(C);\varphi))$ which is a contradiction, and hence $C$ is a complete clutter and the desired result follows.

Now, we prove the claim. Suppose there exists $c\in \tau_2\setminus \tau_1$ such that $\tau_1\cup \{c\}\in \Delta(C)$. We have $c\notin \rho_1$, since $c\notin \tau_1$. Then by Lemma~\ref{linear lemma} part~(b), it follows that $\sigma_1\cup \{c\}\in \Delta(C)$, since $\sigma_1\subseteq  \rho_1$. This is a contradiction, since $c\in \tau_2\subseteq \rho_2$ which implies that $c\in \sigma_2$ because $c\notin A\subseteq \sigma_1$, and so $c\in B\subseteq \rho_1$.
\end{proof}

The following is an straightforward consequence of Theorem~\ref{linear res}.

\begin{Corollary}
Let $C$ be an $m$-clutter and consider the augmentation map
\[
\MC_0(\Delta(C);\varphi)\stackrel{\psi}\longrightarrow J_C \longrightarrow 0.
\]
Then the complex $\MC(\Delta(C);\varphi)\stackrel{\psi}\longrightarrow J_C \longrightarrow 0$ is exact if and only if $C$ is a complete clutter.
\end{Corollary}


\newpage

\begin{thebibliography}{}


\bibitem{BS} D. Bayer, B. Sturmfels, {\em Cellular resolutions of monomial ideals}, J. Reine Angew. Math. 102
(1998), 123-140.


\bibitem{BPS} D. Bayer, I. Peeva, B. Sturmfels, {\em Monomial resolutions}, Math. Res. Lett. 5 (1-2) (1998), 31-46.


\bibitem{BV}  W. Bruns, U. Vetter,  {\em Determinantal rings, Lecture Notes in Mathematics}, Springer, (1988).


\bibitem{D} A. Dokuyucu, {\em Extremal Betti numbers of some classes of binomial edge ideals}, to appear in Math. Reports.


\bibitem{EN} J. A. Eagon, D. G. Northcott, {\em Ideals defined by matrices and a certain complex associated with Them}, Proceedings of the Royal Society of London. Series A, Mathematical and Physical Sciences 269, (1962), no. 1337, 188-204.


\bibitem{Ei} D. Eisenbud, {\em Commutative algebra with a view towards algebraic geometry}, Springer, (1995).


\bibitem{Ei1} D. Eisenbud, {\em The geometry of syzygies (A second course in Commutative Algebra and Algebraic Geometry)}, Springer, (2005).


\bibitem{EHH} V. Ene, J. Herzog, T. Hibi, {\em Cohen-Macaulay binomial edge ideals}, Nagoya Math. J. 204 (2011), 57-68.


\bibitem{EHH1} V. Ene, J. Herzog, T. Hibi, {\em Koszul binomial edge ideals}, Bridging Algebra, Geometry, and Topology, Volume 96 of the series Springer Proceedings in Mathematics and Statistics, (2014), 125-136.


\bibitem{EHHM} V. Ene, J. Herzog, T. Hibi, F. Mohammadi, {\em Determinantal facet ideals}, Michigan Math. J. 62 (2013), 39-57.


\bibitem{EZ} V. Ene, A. Zarojanu, {\em On the regularity of binomial edge ideals}, Math. Nachr. 288,  No. 1 (2015), 19-24.


\bibitem{Fl} G. Fl{\o}ystad, {\em Enriched homology and cohomology modules of simiplicial complexes}, J. Algebr. Comb. 25 (2007), 285-307.


\bibitem{HHHKR} J. Herzog, T. Hibi, F. Hreinsdotir, T. Kahle, J. Rauh, {\em Binomial edge ideals and conditional independence statements},
Adv. Appl. Math. 45 (2010), 317-333.


\bibitem{KS} D. Kiani, S. Saeedi Madani, {\em Binomial edge ideals with pure resolutions}, Collect. Math. 65 (2014), 331-340.


\bibitem{KS1} D. Kiani, S. Saeedi Madani, {\em Some Cohen-Macaulay and unmixed binomial edge ideals}, to appear in Comm. Algebra, (arXiv:1506.01362).


\bibitem{KS2} D. Kiani, S. Saeedi Madani, {\em The Castelnuovo-Mumford regularity of binomial edge ideals}, arXiv:1504.01403.


\bibitem{KN} S. Klee, I. Novik, {\em From flag complexes to banner complexes}, SIAM J. Discrete Math., 27 (2013), no. 2, 1146-1158.


\bibitem{MM} K. Matsuda, S. Murai, {\em Regularity bounds for binomial edge ideals}, Journal of Commutative Algebra. 5(1) (2013), 141-149.


\bibitem{Mo} F. Mohammadi, {\em Prime splittings of determinantal ideals}, arXiv:1208.2930.


\bibitem{Ne} E. Nevo, {\em Remarks on missing faces and lower bounds on face numbers}, Electronic J. Combin. 16(2) (the Bj\"{o}rner Festschrift volume) (2009), R8.

\bibitem{O} M. Ohtani, {\em Graphs and ideals generated by some 2-minors}, Comm. Algebra. 39 (2011), 905-917.


\bibitem{SK} S. Saeedi Madani, D. Kiani, {\em Binomial edge ideals of graphs}, Electronic J. Combin. 19(2) (2012), $\sharp$ P44.


\bibitem{SK1} S. Saeedi Madani, D. Kiani, {\em On the binomial edge ideal of a pair of graphs}, The Electronic Journal of Combinatorics. 20(1) (2013), $\sharp$ P48.


\bibitem{SZ} P. Schenzel, S. Zafar, {\em Algebraic properties of the binomial edge ideal of a complete bipartite graph},  An. St. Univ. Ovidius Constanta, Ser. Mat. 22(2) (2014), 217-237.


\bibitem{ZZ} Z. Zahid, S. Zafar, {\em On the Betti numbers of some classes of binomial edge ideals}, The Electronic Journal of Combinatorics. 20(4) (2013), $\sharp$ P37.

\end{thebibliography}
 \end{document}